\documentclass[11pt,reqno]{amsart}

\usepackage{amsfonts}
\usepackage{latexsym}
\usepackage{amsmath}
\usepackage{amssymb}
\usepackage[usenames]{color}

\newcommand{\C}{\mathbb{C}}

\newcommand{\Z}{\mathbb{Z}}

\newtheorem{theorem}{Theorem}[section]

\newtheorem{corollary}[theorem]{Corollary}
\newtheorem{definition}[theorem]{Definition}
\newtheorem{example}[theorem]{Example}
\newtheorem{proposition}[theorem]{Proposition}

\begin{document}


\noindent Accepted for publication in the Proceedings of the Edinburgh Mathematical Society

\vskip 5mm

\title{Value distribution and linear operators}
\author{Rodney Halburd}
\address{Department of Mathematics, University College London,
Gower Street, London WC1E 6BT, UK} \email{r.halburd@ucl.ac.uk}
\author{Risto Korhonen}
\address{Department of Physics and Mathematics, University of Eastern Finland, P.O. Box 111,
FI-80101 Joensuu, Finland}
\email{risto.korhonen@uef.fi}
\thanks{The work reported here was supported in part by EPSRC grant number EP/I013334/1 and the Academy of Finland Grant \#112453, \#118314 and \#210245}

\begin{abstract}
Nevanlinna's second main theorem is a far-reaching generalisation of Picard's Theorem concerning the value distribution of an arbitrary meromorphic function $f$.  The theorem takes the form of an inequality containing a ramification term in which the zeros and poles of the derivative $f'$ appear.  In this paper we show that a similar result holds for special subfields of meromorphic functions where the derivative is replaced by a more general linear operator, such as higher-order differential operators and differential-difference operators.  We subsequently derive generalisations of Picard's Theorem and the defect relations.
\end{abstract}


\maketitle

\numberwithin{equation}{section}

\section{Introduction}
Nevanlinna theory studies the value distribution of meromorphic functions.  Central to the classical theory is the second main theorem together with the related defect relations, which are powerful generalisations of Picard's Theorem.  Nevanlinna's second main theorem uses the distribution of points in the closed disc $|z|\le r$ at which a meromorphic function $f$ takes certain prescribed values to bound the Nevanlinna characteristic $T(r,f)$.  The theorem incorporates information from the distribution of zeros and poles of the derivative $f'$ through a {\em ramification term} which ensures that when we count the preimages of the prescribed values we can ignore multiplicities.

In this paper we derive an analogue of the second main theorem in which the derivative $f\mapsto f'$ is replaced by an arbitrary linear operator
$f\mapsto L(f)$ on any subfield $\mathcal{N}$ of the space of meromorphic functions such that $m(r,L(f)/f)=o(T(r,f))$ for $r$ in a large subset of 
$(0,\infty)$.  Examples of such operators are the derivative $f(z)\mapsto f'(z)$, the shift $f(z)\mapsto f(z+c)$ and the $q$-difference operator
$f(z)\mapsto f(qz)$ as well as combinations such as
$f(z)\mapsto f''(z+c_1)+f(z)-2f'(qz+c_2)$.  We derive a generalisation of Picard's Theorem, which in essence says that if there are enough  functions $a_j\in\ker(L)$ that are small compared with $f$, then $L(f)$ is identically zero.
Analogues of the defect relations are given and several examples are used to illustrate the strength of the results obtained.

The work presented here extends earlier work on the shift \cite{halburdk:06AASFM} and $q$-difference \cite{barnetthkm:07} operators.  Those works in turn grew out of a programme  to use Nevanlinna theory as a tool for detecting and describing difference equations of ``Painlev\'e type'' \cite{ablowitzhh:00,halburdk:07JPA}.
The generalisations described in the present paper are motivated by preliminary studies of differential delay equations such as
\begin{equation}
\label{quispelcs}
\alpha f(z)+\beta f'(z)=f(z)\left[ f(z+1)-f(z-1) \right],
\end{equation}
where $\alpha$ and $\beta$ are constants,
which was obtained by Quispel, Capel and Sahadevan \cite{quispelcs:92} as a reduction of the Kac-van Moerbeke equation.  Equation \eqref{quispelcs} is known to have a continuum limit to the first Painlev\'e equation.

\section{A second main theorem with general ramification-type term}\label{2ndsection}

Let $R\in(0,\infty]$, and let $\mathcal{M}_R$ be the set of all meromorphic functions in $D_R=\{z\in\C:|z|<R\}$. Let $f\in\mathcal{M}_R$ and $a\in\C$. Then there exist $m\in\Z$ and $\tilde f\in\mathcal{M}_R$ such that
    $$
    f(z)=(z-a)^m \tilde f(z),
    $$
where $\tilde f(a)\in\C\setminus\{0\}$. As in \cite{cherryy:01}, we say that $\tilde f(a)$ is the \textit{initial Laurent coefficient} of $f$ at $a$, and define
    $$
    \textrm{ilc}(f,a)=\tilde f(a).
    $$
The following theorem reduces to the second main theorem in $D_R$ by choosing $g=f'$ in the ramification type term $N_{g}(r,f)$, and by using the lemma on the logarithmic derivative.

\begin{theorem}\label{NTmainthm}
Let $R\in(0,\infty]$, and let $g\in\mathcal{M}_R\setminus\{0\}$. If $a_1,\ldots,a_q$ are $q\geq1$ different elements of $\mathcal{M}_R$ and if $f\in\mathcal{M}_R\setminus\{a_1,\ldots,a_q\}$, then
	\begin{equation}\label{mainineq}
	\begin{split}
	(q-1)T(r,f)+N_{g}(r,f)&\leq N(r,f)+\sum_{j=1}^q N\left(r,\frac{1}{f-a_j}\right)+\mathcal{R}(r,f,g),
	\end{split}
	\end{equation}	
where
	$$
	N_{g}(r,f) = 2N(r,f)-N(r,g)+N\left(r,\frac{1}{g}\right)
	$$
and
	\begin{equation*}
	\begin{split}
	\mathcal{R}(r,f,g) &= \sum_{m=1}^q T(r,a_m) +\frac{1}{2\pi}\int_0^{2\pi}
	 \log\left(\sum_{m=1}^q\left|\frac{g(re^{i\theta})}{f(re^{i\theta})-a_{m}(re^{i\theta})}\right|\,d\theta\right) \\&\quad+\frac{q-1}{2\pi}\int_0^{2\pi}\log^+
	 \frac{2}{l(re^{i\theta})}\,d\theta+ (q-1)\log 2\\ &\quad
	+\sum_{m=1}^q \log|\emph{ilc}(f-a_m,0)|-\log|\emph{ilc}(g,0)|
	\end{split}
	\end{equation*}
with
	$$
	l(re^{i\theta})=\min_{1\leq i<j\leq q}|a_i(re^{i\theta})-a_j(re^{i\theta})|.
	$$
\end{theorem}

\begin{proof}
We follow parts of the proof of the second main theorem in \cite{cherryy:01}. Let $z=re^{i\theta}$ and choose $s\in\{1,\ldots,q\}$ such that $|f(z)-a_{s}(z)|\leq|f(z)-a_m(z)|$ for all $m\in\{1,\ldots,q\}$. If $q\geq 2$ then the number $s$ generally depends on $z$. In this case, moreover,
	\begin{equation*}
	\begin{split}
	l(z)& \leq \min_{1\leq j\leq q \atop j\not=s}|a_{s}(z)-a_j(z)|
	\leq \min_{1\leq j\leq q \atop j\not=s}\left\{|f(z)-a_{s}(z)| +|f(z)-a_j(z)|\right\}\\
	&\leq |f(z)-a_{s}(z)| +|f(z)-a_m(z)|
	\leq 2|f(z)-a_m(z)|
	\end{split}
	\end{equation*}
for all $m\in\{1,\ldots,q\}\setminus\{s\}$, and so
	\begin{equation*}\label{xam}
	\begin{split}
	\log|f(z)-a_m(z)| &=\log^+|f(z)-a_m(z)|-\log^+\frac{1}{|f(z)-a_m(z)|}\\
	&\geq \log^+|f(z)-a_m(z)| - \log^+\frac{2}{l(z)}.
	\end{split}
	\end{equation*}
Therefore,
	\begin{equation}\label{mineq}
	\begin{split}
	(q-1)\log^+|f(z)| &\leq \sum_{m=1, \atop
	 m\not=s}^q\left(\log^+|f(z)-a_m(z)| +\log^+|a_m(z)|\right)
	 +(q-1)\log 2\\
	 &\leq \sum_{m=1, \atop
	 m\not=s}^q\log|f(z)-a_m(z)|
	 +\sum_{m=1, \atop
	 m\not=s}^q\log^+|a_m(z)|\\ &\quad+
	(q-1)\log^+
	 \frac{2}{l(z)}+(q-1)\log 2.\\
	\end{split}
	\end{equation}
Furthermore, we have
	\begin{equation}\label{mineq2}
	\begin{split}
	\sum_{m=1, \atop
	 m\not=s}^q\log|f(z)-a_m(z)|
	 &= \sum_{m=1}^q\log|f(z)-a_m(z)| - \log |g(z)|
	 +\log\left|\frac{g(z)}{f(z)-a_{s}(z)}\right|.
	\end{split}
	\end{equation}
The last term on the right-hand side of \eqref{mineq2} contains a term depending on $s$, which, in turn, depends on $z$. We will now make equation \eqref{mineq2} independent of $s$ by adding a number of terms to its right side. This yields the inequality
	\begin{equation}\label{mineq3}
	\begin{split}
	\sum_{m=1, \atop
	 m\not=s}^q\log|f(z)-a_m(z)| &\leq
	  \sum_{m=1}^q\log|f(z)-a_m(z)| - \log |g(z)|
	 +\log\left(\sum_{m=1}^q\left|\frac{g(z)}{f(z)-a_{m}(z)}\right|\right),
	\end{split}
	\end{equation}
which also holds if $q=1$. We may now average over all directions $\theta\in[0,2\pi)$. By combining \eqref{mineq} and \eqref{mineq3} we then have
	\begin{equation}\label{eq4}
	\begin{split}
	(q&-1)m(r,f) \\ &\leq \frac{1}{2\pi}\int_0^{2\pi}\sum_{m=1}^q\log|f(re^{i\theta})-a_m(re^{i\theta})|d\theta
	- \frac{1}{2\pi}\int_0^{2\pi} \log |g(re^{i\theta})|d\theta \\ &\quad
	 +\frac{1}{2\pi}\int_0^{2\pi}
	 \log\left(\sum_{m=1}^q\left|\frac{g(re^{i\theta})}{f(re^{i\theta})-a_{m}(re^{i\theta})}\right|\right)d\theta
	+\frac{1}{2\pi}\int_0^{2\pi}\sum_{m=1}^q\log^+|a_m(re^{i\theta})|d\theta \\&\quad+
	\frac{(q-1)}{2\pi}\int_0^{2\pi}\log^+
	 \frac{2}{l(re^{i\theta})}d\theta+(q-1)\log 2.\\
	\end{split}
	\end{equation}
The Jensen's formula implies that
	\begin{equation}\label{artinwhaples2}
	\frac{1}{2\pi} \int_0^{2\pi}\log|h(re^{i\theta})|d\theta
	= N\left(r,\frac{1}{h}\right)-N(r,g) + \log|\textrm{ilc}(h,0)|
	\end{equation}	
for any $h\in \mathcal{M}_R\setminus\{0\}$. By combining \eqref{eq4} and \eqref{artinwhaples2}, we have	
	\begin{equation}\label{eq5}
	\begin{split}
	(q&-1)m(r,f) \\ &\leq   \sum_{m=1}^q N\left(r,\frac{1}{f-a_m}\right)-
	q N(r,f)+  N(r,g)-N\left(r,\frac{1}{g}\right) \\
	&\quad +\sum_{m=1}^q T(r,a_m)
	 +\frac{1}{2\pi}\int_0^{2\pi}
	 \log\left(\sum_{m=1}^q\left|\frac{g(re^{i\theta})}{f(re^{i\theta})-a_{m}(re^{i\theta})}\right|\right)d\theta \\&\quad+\frac{q-1}{2\pi}\int_0^{2\pi}\log^+
	 \frac{2}{l(re^{i\theta})}d\theta+ (q-1)\log 2\\ &\quad
	+\sum_{m=1}^q \log|\textrm{ilc}(f-a_m,0)|-\log|\textrm{ilc}(g,0)|.
	\end{split}
	\end{equation}
The assertion follows by adding $(q-1)N(r,f)$ to both sides of \eqref{eq5}.
\end{proof}

\section{Linear differential operator}

We will now specialize Theorem \ref{NTmainthm} by choosing $g$ to be a linear differential operator applied to the function $f$, meromorphic in the complex plane. Let
    $$
    \mathcal{S}(f)=\{h\in\mathcal{M}_\infty:T(r,h)=o(T(r,f))\mbox{ as } r\to \infty,\, r\not\in E\}
    $$
where  $E\subset (0,\infty)$ is a set having finite linear measure, and set $\mathcal{M}:=\mathcal{M}_\infty$ for brevity. By choosing
	$$
	g=f^{(n)}+\alpha_{n-1}f^{(n-1)}+\cdots+\alpha_1 f' + \alpha_0 f,
	$$
in Theorem \ref{NTmainthm}, where $\alpha_0,\ldots,\alpha_{n-1}\in \mathcal{S}(f)$, we can show that the remainder term in \eqref{mainineq} is small. This yields the following generalization of Nevanlinna's second main theorem for general linear differential operators.

\begin{theorem}\label{NTmainthmlin}
Let $L:\mathcal{M}\to\mathcal{M}$ be given by
	\begin{equation}\label{linop}
	L(h)=h^{(n)}+\alpha_{n-1}h^{(n-1)}+\cdots+\alpha_1 h' + \alpha_0 h
	\end{equation}
where $\alpha_0,\ldots,\alpha_{n-1}\in \mathcal{S}(f)$, and let $f\in\mathcal{M}\setminus\ker(L)$. If $a_1,\ldots,a_q$ are $q\geq1$ different elements of $\ker(L) \cap \mathcal{S}(f)$, then
	\begin{equation}\label{mainineqlin}
	\begin{split}
	(q-1)T(r,f)+N_{L(f)}(r,f)&\leq N(r,f)+\sum_{j=1}^q N\left(r,\frac{1}{f-a_j}\right)+S(r,f),
	\end{split}
	\end{equation}	
where
	$$
	N_{L(f)}(r,f) = 2N(r,f)-N(r,L(f))+N\left(r,\frac{1}{L(f)}\right)
	$$
and
	\begin{equation*}
	S(r,f) = o(T(r,f))
	\end{equation*}
as $r\to \infty$ outside a set of finite linear measure.
\end{theorem}

\begin{proof}
It is sufficient to estimate the remainder $\mathcal{R}(r,f,g)$ in Theorem \ref{NTmainthm}. Since $a_1,\ldots,a_q\in \mathcal{S}(f)$, it follows that
	\begin{equation*}
	\sum_{m=1}^q T(r,a_m)+\frac{q-1}{2\pi}\int_0^{2\pi}\log^+
	 \frac{2}{l(re^{i\theta})}\,d\theta = o(T(r,f))
	\end{equation*}
as $r\to \infty$ such that $r\not\in E$. Since the term
	\begin{equation*}
	(q-1)\log 2	+\sum_{m=1}^q \log|\textrm{ilc}(f-a_m,0)|-\log|\textrm{ilc}(L(f),0)|
	\end{equation*}
is constant, it only remains to show that
	\begin{equation}\label{small}
	 \frac{1}{2\pi}\int_0^{2\pi}\log\left(\sum_{m=1}^q\left|\frac{(L\circ f)(re^{i\theta})}{f(re^{i\theta})-a_{m}(re^{i\theta})}\right|\,d\theta\right) = o(T(r,f))
	\end{equation}
when $r\to \infty$ outside of the set $E$. Since $a_j\in\ker(L)$ and $L$ is linear, it follows that
	$$
	L(f)=L(f)-L(a_j)=L(f-a_j)
	$$
for all $j=1,\ldots,q$. Therefore
	\begin{equation}\label{small2}
	\begin{split}
	\frac{1}{2\pi}\int_0^{2\pi} \log\left(\sum_{m=1}^q\left|\frac{(L\circ f)(re^{i\theta})}{f(re^{i\theta})-a_{m}(re^{i\theta})}\right|\,d\theta\right)
	\leq \sum_{j=1}^q m\left(r,\frac{L(f-a_j)}{f-a_j}\right) +O(1).
	\end{split}
	\end{equation}
Since $a_j\in \mathcal{S}(f)$ for all $j=1,\ldots,q$, the equation \eqref{small} follows by combining \eqref{small2} with the lemma on the logarithmic derivative.
\end{proof}

A theorem similar to Theorem \ref{NTmainthmlin}, but for meromorphic functions in a disc, can be stated and proved in an almost identical way to the proof above.

\begin{definition}
Let $L:\mathcal{M}\to\mathcal{M}$ be linear, and let $a\in\ker{L}$. If the preimage of $a$ under $f$ (understood as a multiset where each point is repeated the number of times indicated by its multiplicity) is contained in the preimage of $0$ under $L(f)$, then $a$ is said to be \textit{$(L,f)$-exceptional}.
\end{definition}

Note that functions with empty preimages are automatically $(L,f)$-exceptional for any $L$. Therefore, in particular, Picard exceptional values of $f\in\mathcal{M}$ are 
$(f',f)$-exceptional. Hence the following two corollaries of Theorem \ref{NTmainthmlin} are generalizations of Picard's Theorem.

\begin{corollary}\label{picardcormero}
Let $f$ be meromorphic, and let $L$ be an $n^\textrm{th}$ order linear differential operator with coefficients in $\mathcal{S}(f)$. If $a_1,\ldots,a_{n+2}\in\mathcal{S}(f)$ are $n+2$ distinct $(L,f)$-exceptional functions, then $L(f)=0$.
\end{corollary}

\begin{proof}
If $L(f)\not=0$, then Theorem \ref{NTmainthmlin} yields
	\begin{equation}\label{Trg}
	\begin{split}
    (n+1)T(r,f) &\leq N(r,L(f))-N(r,f)+\sum_{j=1}^{n+2} N\left(r,\frac{1}{f-a_j}\right) - N\left(r,\frac{1}{L(f)}\right) + o(T(r,f))\\
	\end{split}
	\end{equation}
where $r\to \infty$ outside of an exceptional set $E$. Now, since $L(f)$ is of order $n$, it follows that
    \begin{equation}\label{NLf}
    N(r,L(f))-N(r,f) \leq n T(r,f)
    \end{equation}
for all $r>0$. Moreover, since $a_1,\ldots,a_{n+2}$ are $(L,f)$-exceptional, we have
	$$
	\sum_{j=1}^{n+2} N\left(r,\frac{1}{f-a_j}\right) \leq N\left(r,\frac{1}{L(f)}\right),
	$$
which, together with \eqref{NLf} and \eqref{Trg}, yields a contradiction. Hence $L(f)=0$.
\end{proof}

The following example shows that $n+2$ cannot be replaced by $n+1$ in Corollary~\ref{picardcormero} in the case $n=2$.

\begin{example}\label{jacobiexample}\rm
The Jacobi elliptic function $f(z)=\textrm{sn}(z,k)$ satisfies the differential equation
    \begin{equation}\label{jacobieq}
    f''= 2k^2 f^3 - (1+k^2)f
    \end{equation}
where $k\in(0,1)$ is the elliptic modulus. By \eqref{jacobieq} it follows that $L(f)=f''$ vanishes if and only if $f$ attains one of the values in the set
    \begin{equation}\label{values}
    \left\{0,\frac{1}{\sqrt{2} k}\sqrt{1+k^2} ,-\frac{1}{\sqrt{2} k}\sqrt{1+k^2} \right\}.
    \end{equation}
Therefore the elements of \eqref{values} are three distinct $(L,f)$-exceptional functions, but clearly $L(f)\not= 0$.
\end{example}

In the case of entire functions the order of the linear operator does not affect the number of required target functions.

\begin{corollary}\label{picardcor}
Let $f$ be entire, and let $L$ be a linear differential operator with coefficients in $\mathcal{S}(f)$. If $a\in\mathcal{S}(f)$ and $b\in\mathcal{S}(f)$ are two distinct $(L,f)$-exceptional functions of $f$, then $L(f)=0$.
\end{corollary}

\begin{proof}
Assume that $L(f)\not=0$. Then, since $f$ is entire, Theorem \ref{NTmainthmlin} yields
	\begin{equation}\label{Trge}
	\begin{split}
	T(r,f) &\leq N\left(r,\frac{1}{f-a}\right)+N\left(r,\frac{1}{f-b}\right) - N\left(r,\frac{1}{L(f)}\right) + o(T(r,f))\\
	\end{split}
	\end{equation}
where $r\to \infty$ outside of an exceptional set $E$. Since $a$ and $b$ are $(L,f)$-exceptional, it follows that
	$$
	N\left(r,\frac{1}{f-a}\right)+N\left(r,\frac{1}{f-b}\right) \leq N\left(r,\frac{1}{L(f)}\right),
	$$
and so inequality \eqref{Trge} leads to a contradiction. Thus $L(f)=0$.
\end{proof}

Corollary \ref{picardcor} says that if there are enough points where an entire function $f$ looks like a solution of a linear differential equation, then $f$ must indeed be a solution of the equation. This statement can be made more precise by introducing deficiencies with respect to linear differential operators. Towards this end, let $a\in\ker(L)$, and let
    \begin{equation}\label{countingfunction}
    N{|_{ f=a}} \left(r,\frac{1}{L(f)}\right)
    \end{equation}
be the integrated counting function for those zeros of $L(f)$, where simultaneously $f=a$. Note that even though it is required that $f=a$ for a zero of $L(f)$ to be counted in \eqref{countingfunction}, multiplicities are counted only according to multiplicities of zeros of $L(f)$. Multiplicities of the $a$-points of $f$ do not contribute to \eqref{countingfunction}. Using \eqref{countingfunction} we define \textit{the index of multiplicity with respect to} $L(f)$ by
    \begin{equation}\label{def1}
    \theta_{L,f}(a):= \liminf_{r\to \infty} \frac{1}{T(r,f)} \left(N{|_{ f=a}} \left(r,\frac{1}{L(f)}\right)\right),
    \end{equation}
which reduces to the usual index of multiplicity $\theta(a,f)$ when $L(f)=f'$. 
For a particular operator $L$ we will sometimes write $ \theta_{L(f)}$ instead of $\theta_{L,f}$ for brevity.
We also define
    \begin{equation}\label{def2}
    \theta_{L,f}(\infty):= \liminf_{r\to \infty} \frac{2N(r,f)-N\left(r,L(f)\right)}{T(r,f)},
    \end{equation}
which is similarly a generalization of $\theta(\infty,f)$. Note that even though $0\leq\theta(a,f)\leq 1$, there is in principle no upper bound for $\theta_{L,f}(a)$ implied by the definition itself. Also,
    $$
    -k+1\leq \theta_{L,f}(\infty) \leq 1,
    $$
where $k$ is the degree of the highest derivative of $L(f)$. By following the standard proof of Nevanlinna's deficiency relation (see, for example, \cite[pp.~43--44]{hayman:64}) and using definitions \eqref{def1} and \eqref{def2}, we obtain the following consequence of Theorem \ref{NTmainthmlin}.

\begin{corollary}\label{rel}
Let $f\not\in\ker(L)$ be a meromorphic function. Then $\theta_{L,f}(a)=0$
except for at most countably many $a\in\ker(L)\cap \mathcal{S}(f)$, and
    \begin{equation}\label{relsum}
    \sum_{a} \left(\delta(a,f)+ \theta_{L,f}(a)\right)\leq 2,
    \end{equation}
where the summation is over all elements of the set $(\ker(L)\cap \mathcal{S}(f))\cup\{\infty\}$.
\end{corollary}

Corollary \ref{rel} implies the usual deficiency relation by substituting $L(f)=f'$. The following continuation of example \ref{jacobiexample} shows that the total deficiency sum of \eqref{relsum} can be strictly greater that 2, if the term $\theta_{L,f}(\infty)$ is omitted from \eqref{relsum}. The same example also shows that the upper bound in \eqref{relsum} can be attained.

\begin{example}\label{jacobiexample2}\rm
Let $f(z)=\textrm{sn}(z,k)$ be a solution of \eqref{jacobieq}. In example \ref{jacobiexample} we saw that $L(f)=f''$ vanishes if and only if $f$ takes one of the values in the set \eqref{values}. By a result due to Mohon'ko (see, for example, \cite[Proposition 9.2.3]{laine:93}) and since $f$ also satisfies
    $$
    (f')^2=(1-f^2)(1-k^2f^2),
    $$
it follows that none of values in \eqref{values} are deficient in the usual sense. (Alternatively, this follows by the fact that elliptic functions attain all values in a period parallelogram the same finite number of times, counting multiplicity.) Hence
    $$
    \theta_{f''}(0)=\theta_{f''}\left(\frac{1}{\sqrt{2} k}\sqrt{1+k^2}\right)=\theta_{f''}\left(-\frac{1}{\sqrt{2} k}\sqrt{1+k^2}\right) = 1,
    $$
and so
    $$
    \sum_{a\not=\infty} \left(\delta(a,f)+ \theta_{f''}(a)\right) \geq 3.
    $$
Therefore it follows by Corollary \ref{rel} that $\theta_{f''}(\infty) = -1$, and so, again by the same corollary, we have
    $$
    \sum_{a} \left(\delta(a,f)+ \theta_{f''}(a)\right) =2.
    $$
\end{example}

\section{General linear operator}

By demanding that the operator $L:\mathcal{M}\to\mathcal{M}$ is linear, and that the target functions $a_1,\ldots,a_q$ lie in the intersection of the kernel of $L$ and the field $\mathcal{S}(f)$, we obtain the following theorem, which incorporates and generalizes Nevanlinna's second main theorem in the complex plane, and its difference analogue from \cite{halburdk:06AASFM}, as well as the $q$-difference second main theorem from \cite{barnetthkm:07}.

\begin{theorem}\label{NTmainthm2}
Let $\mathcal{N}$ be a subfield of $\mathcal{M}$  and let $f\in\mathcal{N}\setminus\ker(L)$, where $L:\mathcal{M}\to\mathcal{M}$ is a linear operator such that
	\begin{equation}\label{assumption}
	m\left(r,\frac{L(f)}{f}\right)=o(T(r,f))
	\end{equation}
as $r\to\infty$ outside of an exceptional set $E\subset(0,\infty)$. If $a_1,\ldots,a_q$ are $q\geq1$ different elements of $\ker(L) \cap \mathcal{S}(f)$, then
	\begin{equation}\label{mainineq2}
	\begin{split}
	(q-1)T(r,f)+N_{L(f)}(r,f)&\leq N(r,f)+\sum_{j=1}^q N\left(r,\frac{1}{f-a_j}\right)+S(r,f),
	\end{split}
	\end{equation}	
where
	$$
	N_{L(f)}(r,f) = 2N(r,f)-N(r,L(f))+N\left(r,\frac{1}{L(f)}\right)
	$$
and
	\begin{equation*}
	S(r,f) = o(T(r,f))
	\end{equation*}
as $r\to \infty$ outside of $E$.
\end{theorem}

Theorem \ref{NTmainthm2} implies Nevanlinna's second main theorem in the complex plane by our choosing $L(f)=f'$ and $\mathcal{N}=\mathcal{M}$. By choosing $L(f)=\Delta(f)=f(z+1)-f(z)$ and $\mathcal{N}$ to be the field of meromorphic functions of hyper-order strictly less than one, Theorem \ref{NTmainthm2} reduces into the difference analogue of the second main theorem \cite{halburdk:06AASFM,halburdkt:09}. Similarly, with the choice $L(f)=f(qz)-f(z)$, where $q\in\C\setminus\{0,1\}$, and taking $\mathcal{N}$ to be the field of zero-order meromorphic functions, we obtain the $q$-difference version of the second main theorem \cite{barnetthkm:07}.

Deficiencies can be defined as above using formulas \eqref{def1} and \eqref{def2}, where $L(f)$ is now any differential operator operating on a meromorphic function $f$ such that \eqref{assumption} is satisfied. For instance, if $L(f)=f'(z+1)$, where the hyper-order (or iterated $2$-order) $\varsigma(f)$ of $f$ satisfies $\varsigma(f)=\varsigma<1$, then by the lemma on the logarithmic derivative and its difference analogue \cite{halburdkt:09}, it follows that
    \begin{equation}\label{assumptioncalc}
    \begin{split}
    m\left(r,\frac{L(f)}{f}\right) & = m\left(r,\frac{f'(z+1)}{f(z)}\right) \\
    & \leq m\left(r,\frac{f'(z+1)}{f'(z)}\right) + m\left(r,\frac{f'(z)}{f(z)}\right) \\
    &= o\left(\frac{T(r,f')}{r^{1-\varsigma-\varepsilon}}\right) + O(\log(rT(r,f))) \\
    &= o(T(r,f)),
    \end{split}
    \end{equation}
where $r\to\infty$ outside of an exceptional set of finite logarithmic measure, and we have taken $\varepsilon>0$ such that $\varsigma+\varepsilon<1$. Hence \eqref{assumption} is satisfied, and so Theorem \ref{NTmainthm2} yields a counterpart of Corollary \ref{rel} under the assumption that $f$ is a non-constant meromorphic function of hyper-order strictly less than one. We state this result here as a proposition.

\begin{proposition}\label{rel2}
Let $f$ be a non-constant meromorphic function such that $\varsigma(f)<1$. Then $\theta_{f'(z+1)}(a)=0$
except for at most countably many $a\in\C$, and
    \begin{equation*}
    \sum_{a\in\C\cup\{\infty\}} \left(\delta(a,f)+ \theta_{f'(z+1)}(a)\right)\leq 2.
    \end{equation*}
\end{proposition}

Proposition \ref{rel2} is, of course, just one example of the type of results that can be obtained this way. For instance, if $L(f)=f'(qz+c)$, where $q\in\C\setminus\{0,1\}$, then we can obtain a counterpart of proposition \ref{rel2} under the stricter assumption that $f$ is of zero order by using a $q$-difference analogue of the lemma on the logarithmic derivative \cite{barnetthkm:07} in a calculation similar to \eqref{assumptioncalc}.

Proposition \ref{rel2} implies some rather surprising constraints for possible value distribution patterns of finite-order meromorphic functions.

\begin{example}\rm
Let $a\in\C$, and let $f$ be a meromorphic function of finite order such that all of its $a$-points are simple, and moreover, if $f(z)=a$ then $f'(z+1)=0$ with multiplicity $k\geq p \geq 2$.  Then the Valiron deficiency
    \begin{equation}\label{valironDef}
    \Delta(a,f)=1-\liminf_{r\to\infty}\frac{N\left(r,\frac{1}{f-a}\right)}{T(r,f)}
    \end{equation}
satisfies
    \begin{equation}\label{concl}
    \Delta(a,f)\geq 1-\frac{2}{p}.
    \end{equation}
Namely, by \eqref{valironDef} it follows that
    \begin{equation}\label{ineq1}
    T(r,f) \leq \frac{1}{1-\Delta(a,f)-\varepsilon} N\left(r,\frac{1}{f-a}\right),
    \end{equation}
where $\varepsilon>0$ and $r$ is sufficiently large. By the assumption on the locations or zeros of $f$ and its derivative function, we have
    \begin{equation}\label{ineq2}
    N{|_{ f=a}} \left(r,\frac{1}{f'(z+1)}\right) \geq p N\left(r,\frac{1}{f-a}\right),
    \end{equation}
and so, by \eqref{ineq1} and \eqref{ineq2}, it follows that
    \begin{equation}\label{theta1}
    \theta_{f'(z+1)}(a) \geq (1-\Delta(a,f)-\varepsilon)p.
    \end{equation}
On the other hand, since
    $$
    N(r,f'(z+1)) \leq 2N(r,f(z+1)) \leq 2N(r,f(z))+o\left(\frac{N(r,f)}{r^{1-\varepsilon}}\right)
    $$
for all $r$ outside of an exceptional set of finite logarithmic measure by \cite[Lemma 8.3]{halburdkt:09} (see also \cite[Lemma 2.1]{halburdk:07JPA} and \cite[Theorem 2.2]{chiangf:08}), it follows that $\theta_{f'(z+1)}(\infty)\geq 0$. Therefore proposition \ref{rel2} yields
    \begin{equation}\label{theta2}
    \theta_{f'(z+1)}(a) \leq 2.
    \end{equation}
Inequality \eqref{concl} follows by combining \eqref{theta1} and \eqref{theta2}, and by letting $\varepsilon$ tend to zero. Meromorphic functions satisfying the assumptions on the $a$-points of $f$ and the zeros of $f'$ can be constructed, for instance, by using Hadamard products.
\end{example}

The next, final example shows that certain elliptic functions are maximally deficient with respect to the second order linear difference operator. In the same way as in example \ref{jacobiexample}, the maximal deficiency sum over all finite targets turns out to be equal to three, instead of the usual two.

\begin{example}\rm
The autonomous form of the difference Painlev\'e II equation,
    \begin{equation}\label{mcmillaneq}
    f(z+1)+f(z-1)=\frac{\alpha f(z)+\beta}{1-f(z)^2},\qquad \alpha,\beta\in\C,
    \end{equation}
which is known as the McMillan map \cite{mcmillan:71}, can be solved in terms of elliptic functions.  By writing \eqref{mcmillaneq} in the form
    $$
     f(z+1)-2f(z)+f(z-1)=\frac{2(f-\gamma_1)(f-\gamma_2)(f-\gamma_3)}{1-f(z)^2},
    $$
where $\gamma_1$, $\gamma_2$ and $\gamma_3$ are the roots of the equation $2x^3+(\alpha-2)x+\beta=0$, we can see that $\Delta^2 f(z)=0$ if and only if $f$ attains one of the values $\gamma_1$, $\gamma_2$ or $\gamma_3$. Suppose that $\alpha$ and $\beta$ are chosen so that $\gamma_1$, $\gamma_2$ and $\gamma_3$ are distinct.
Now, similarly as in example~\ref{jacobiexample}, we have
    $$
    \theta_{\Delta^2 f}(\gamma_1)=\theta_{\Delta^2 f}(\gamma_2)=\theta_{\Delta^2 f}(\gamma_3) = 1,
    $$
and so
    $$
    \sum_{a\in\C} \left(\delta(a,f)+ \theta_{\Delta^2 f}(a)\right) = 3
    $$
and $\theta_{\Delta^2 f}(\infty) = -1$. Thus the maximal deficiency sum
    $$
    \sum_{a\in\C\cup\{\infty\}} \left(\delta(a,f)+ \theta_{\Delta^2 f}(a)\right) =2
    $$
is attained.
\end{example}

\bibliographystyle{amsplain}

\def\cprime{$'$}


\begin{thebibliography}{10}

\bibitem{ablowitzhh:00}
M.~J. Ablowitz,  R.~Halburd, B.~Herbst, 
\emph{On the extension of the Painlev\'e property to difference equations},
Nonlinearity \textbf{13} (2000), 889--905.

\bibitem{barnetthkm:07}
D.~Barnett, R.~G. Halburd, R.~J. Korhonen, and W.~Morgan, \emph{Nevanlinna theory for the $q$-difference operator and meromorphic solutions of q-difference equations}, Proc. Roy. Soc. Edin.,
  Sect. A, Math. \textbf{137} (2007), 457--474.

\bibitem{cherryy:01}
W.~Cherry and Z.~Ye, \emph{{N}evanlinna's theory of value distribution},
  Springer-Verlag, Berlin, 2001.

\bibitem{chiangf:08}
Y.~M. Chiang and S.~J. Feng, \emph{On the {N}evanlinna characteristic of
  $f(z+\eta)$ and difference equations in the complex plane}, Ramanujan J.
  \textbf{16} (2008), no.~1, 105--129.

\bibitem{halburdkt:09}
R.~G. Halburd, R.~Korhonen, and K.~Tohge, \emph{Holomorphic curves with shift-invariant hyperplane preimages}, arXiv:0903.3236 (2009).  To appear in 
  Trans. Amer. Math. Soc.

\bibitem{halburdk:06AASFM}
R.~G. Halburd and R.~J. Korhonen, \emph{Nevanlinna theory for the difference
  operator}, Ann. Acad. Sci. Fenn. Math. \textbf{31} (2006), 463--478.

\bibitem{halburdk:07JPA}
\bysame, \emph{Meromorphic solutions of difference equations, integrability and
  the discrete {P}ainlev\'e equations}, J. Phys. A: Math. Theor. \textbf{40}
  (2007), R1--R38.

\bibitem{hayman:64}
W.~K. Hayman, \emph{Meromorphic functions}, Clarendon Press, Oxford, 1964.

\bibitem{laine:93}
I.~Laine, \emph{{N}evanlinna theory and complex differential equations}, Walter
  de Gruyter, Berlin, 1993.

\bibitem{mcmillan:71}
E.~M. McMillan, \emph{A problem in the stability of periodic systems}, Topics
  in modern physics, a tribute to E.V. Condon (E.~Brittin and H.~Odabasi,
  eds.), Colorado Assoc. Univ. Press, Boulder, Colorado, 1971, pp.~219--244.
  
\bibitem{quispelcs:92}G.~R.~W. Quispel, H.~W. Capel and R. Sahadevan,
\emph{Continuous symmetries of differential-difference equations: the Kac-van Moerbeke equation and Painlev\'e reduction},
Phys. Lett. A \textbf{170} (1992), 379--383.

\end{thebibliography}

\end{document}